\documentclass[11pt,oneside,reqno,english]{amsart}
\usepackage[T1]{fontenc}
\usepackage[latin9]{inputenc}
\usepackage{color}
\usepackage{babel}
\usepackage{amstext}
\usepackage{amsthm}
\usepackage{amssymb}
\usepackage{setspace}
\usepackage[unicode=true,pdfusetitle,
 bookmarks=true,bookmarksnumbered=false,bookmarksopen=false,
 breaklinks=true,pdfborder={0 0 0},pdfborderstyle={},backref=false,colorlinks=true]
 {hyperref}

\makeatletter
\theoremstyle{definition}
\newtheorem*{defn*}{\protect\definitionname}
\theoremstyle{plain}
\newtheorem{thm}{\protect\theoremname}
\theoremstyle{plain}
\newtheorem{lem}[thm]{\protect\lemmaname}
\theoremstyle{plain}
\newtheorem{cor}[thm]{\protect\corollaryname}

\usepackage{calrsfs}
\usepackage{graphicx}
\usepackage{color}
\usepackage{perpage}
\usepackage{upquote}
\usepackage{bbm}
\usepackage[shortlabels]{enumitem}
\usepackage{breakurl}
\usepackage{stmaryrd} 
\usepackage{extarrows} 

\MakePerPage{footnote}
\gdef\SetFigFontNFSS#1#2#3#4#5{} 
\gdef\SetFigFont#1#2#3#4#5{} 
\def\clap#1{\hbox to 0pt{\hss#1\hss}}

\DeclareMathOperator{\lmix}{dmix}

\DeclareMathOperator{\mix}{tvmix}
\DeclareMathOperator{\imix}{imix}
\DeclareMathOperator{\ctvmix}{ctvmix}
\DeclareMathOperator{\ectvmix}{etvmix}

\definecolor{myblue}{rgb}{0.09,0.32,0.44} 
\hypersetup{pdfborder={0 0 0},pdfborderstyle={},colorlinks=true,linkcolor=myblue,citecolor=myblue,urlcolor=blue}

\theoremstyle{remark}
\newtheorem*{qst*}{Question}
\newtheorem*{rmrks*}{Remarks}

\newlength{\tempindent} 
\newcommand{\lazyenum}{
\setlength{\tempindent}{\parindent} 
\begin{enumerate}[leftmargin=0cm,itemindent=0.7cm,labelwidth=\itemindent,labelsep=0cm,align=left,label=\arabic*)]
\setlength{\parskip}{\smallskipamount}
\setlength{\parindent}{\tempindent}
}

\makeatletter
\renewcommand{\andify}{%
  \nxandlist{\unskip, }{\unskip{} \@@and~}{\unskip{} \@@and~}}
\def\author@andify{%
  \nxandlist {\unskip ,\penalty-1 \space\ignorespaces}%
    {\unskip {} \@@and~}%
    {\unskip \penalty-2 \space \@@and~}%
}
\let\@wraptoccontribs\wraptoccontribs
\makeatother

\ifpdf
\def\afs#1#2{\href{#1}{\nolinkurl{#2}}}
\else
\def\afs#1#2{\burlalt{#1}{#2}}
\fi

\allowdisplaybreaks

\makeatother

\providecommand{\corollaryname}{Corollary}
\providecommand{\definitionname}{Definition}
\providecommand{\lemmaname}{Lemma}
\providecommand{\theoremname}{Theorem}

\begin{document}
\title{Comparing with octopi}
\author{Gil Alon and Gady Kozma}
\address{GA: Department of Mathematics and Computer Science, The Open University
of Israel, 4353701 Raanana, Israel}
\address{GK: Department of Mathematics and Computer Science, The Weizmann Institute
of Science, 76100 Rehovot, Israel.}
\begin{abstract}
Operator inequalities with a geometric flavour have been successful
in studying mixing of random walks and quantum mechanics. We suggest
a new way to extract such inequalities using the octopus inequality
of Caputo, Liggett and Richthammer. 
\end{abstract}

\maketitle

\section{Introduction}

This note concerns itself with random walks on the symmetric group
$S_{n}$ generated by sets of transpositions, i.e.\ elements which
exchange some $i\ne j$ in $\{1,\dotsc,n\}$ and keep the others as
are. Draw a graph $G$ whose vertices are $\{1,\dotsc,n\}$ and with
an edge between $i$ and $j$ if the transposition $(i,j)$ is in
our set of generators. The random walk on $S_{n}$ with such a set
of generators is called ``the interchange process on $G$'' and
is an object of interest in interacting particle systems and quantum
mechanics. Because of the connection to mathematical physics, the
most interesting graphs are those that have a geometric structure.
For example, boxes in 2 or 3 dimensional grid.

Geometric graphs are difficult to attack using purely algebraic methods,
so some analysis is necessary. One idea that was discovered a number
of times independently is an operator inequality approach. Roughly,
it goes as follows: Take the standard ``multi-commodity flow'' estimate
for the mixing time \cite{S92}. Replace each inequality of numbers
in the proof with a corresponding inequality of operators (a very
quantum idea!). One gets an inequality comparing the generator of
the interchange process on $G$ to the generator of the interchange
process on the complete graph $K_{n}$. This last object can be attacked
purely algebraically. See \cite{Q92,DSC93,CGS15} \textemdash{} to
the best of our knowledge, all three groups of authors developed this
idea independently. 

From a different starting point, in 1992 David Aldous made a general
conjecture about the second eigenvalue of the interchange process.
Handjani and Jungrais suggested an inductive approach and used it
to prove Aldous' conjecture when the graph $G$ is a tree \cite{HJ96}.
Applying the approach of \cite{HJ96} to general graphs required an
inequality, first conjectured in \cite{D10} and in a preliminary
version of \cite{CLR10}, and then proved by Caputo, Liggett and Richthammer
\cite{CLR10} who christened it ``the octopus inequality'' for reasons
which remain a mystery (it certainly does not enjoy 8 arms).

Even though the octopus inequality was devised for a specific approach
to a specific problem, it turns out to have more applications. The
first to use the octopus lemma for something new was Chen \cite{C17}.
Here we give yet another application. We use it to reprove and strengthen
the operator inequality of \cite{Q92,DSC93,CGS15}. For boxes in $\mathbb{Z}^{d}$
our estimate is no better than the existing one, but it is better
for some graphs, and allows to improve results in the literature \cite{J12,O13,MS16}.

Is there something especially interesting in the octopus lemma, allowing
it to be used for various unintended applications? Or is it the case
that any non-trivial inequality for convolution operators on $L^{2}(S_{n})$
is bound to have applications, because they are so hard to get? Only
time will tell.

\section{The main estimate}

For $i\ne j$ we define $\nabla_{ij}=1-(ij)$, an element in $S_{n}[\mathbb{C}]$
(i.e.~formal sums of elements in $S_{n}$ with complex coefficients).
Any element $A=\sum c_{\sigma}\sigma\in S_{n}[\mathbb{C}]$ can be
thought of as a (convolution) operator on $L^{2}(S_{n})$ given by
\[
(Af)(\tau)=\sum_{\sigma}c_{\sigma}f(\sigma\tau).
\]
For a general $A\in S_{n}[\mathbb{C}]$ the corresponding operator
is not necessarily self-adjoint, but it is the case for $\nabla_{ij}$.

A function $w:\binom{\{1,\dotsc,n\}}{2}\to[0,\infty)$ will be called
a ``weight function'' and we will denote $w_{ij}=w_{ji}=w(\{i,j\})$.
For any such $w$ denote 
\[
\Delta_{w}=\sum_{i<j}w_{ij}\nabla_{ij},
\]
again, as either an element of $S_{n}[\mathbb{C}]$ or as an operator
on $L^{2}(S_{n})$. Denote 
\begin{equation}
w_{i}=\sum_{j\ne i}w_{ij}\qquad w_{\text{tot}}=\sum_{i}w_{i}.\label{eq:def tot}
\end{equation}
We will associate a graph $G$ on $n$ vertices with the weight function
$w$ with $w_{ij}=1$ whenever $\{i,j\}$ is an edge and 0 otherwise.
For example,
\[
\Delta_{K_{n}}=\sum_{i<j}\nabla_{ij}.
\]
The operator $\Delta_{K_{n}}$ has been studied extensively using
representation theory and is relatively well-understood. Our main
result is a way to compare $\Delta_{w}$ for a general $w$ to $\Delta_{K_{n}}$.
But the comparison uses the notion of the mixing time. Since this
has multiple definitions (essentially equivalent, granted), let us
state the one we will be using.
\begin{defn*}
Let $w$ be a weight function. Let $R_{t}$ be the Markov chain on
$\{1,\dotsc,n\}$ whose transition probabilities $p(i,j)$ are given
by
\[
p(i,j)=\begin{cases}
\frac{w_{ij}}{2w_{i}} & i\ne j\\
\frac{1}{2} & i=j.
\end{cases}
\]
In other words, the natural discrete time Markov chain associated
with $w$, with an extra $\frac{1}{2}$ laziness. Let $p_{t}(i,j)$
be the probability that $R_{t}=j$, when $R_{0}=i$. Let $\pi(j)=w_{j}/w_{\textrm{tot}}$
be the stationary measure (it is well-known that, under an assumption
of connectivity, it is also the case that $\pi(j)=\lim_{t\to\infty}p_{t}(i,j)$
and in particular the limit exists and is independent of $i$). Denote
\begin{equation}
\lmix w:=\min\{t:\forall i,j\;p_{t}(i,j)>\tfrac{3}{4}\pi(j)\}.\label{eq:def mixing}
\end{equation}
(``dmix'' standing for discrete time mixing. Unfortunately there
will be a few different notions of mixing in this paper, so we need
to be careful). Note that the function $\min_{i,j}p_{t}(i,j)/\pi(j)$
is increasing in $t$, so in fact the condition of the definition
holds for all $t\ge\lmix w$ (this follows from the stationarity of
$\pi$, we skip the easy calculation). Note also that we only require
a lower bound on $p_{t}$. This means that our definition of $\lmix w$
is equivalent to the usual (total variation) mixing time and not to
the (sometimes larger) $l^{2}$ mixing time. We will prove this below
in lemma \ref{lem:hashevakh lealdous}.
\end{defn*}
Going back to $\Delta_{w}$, the result we wish to prove is essentially
$\Delta_{w}\ge(c/n\lmix w)\Delta_{K_{n}}$, but there is an annoying
extra term (which we will denote by $\delta$) which is constant in
all cases of interest, but not always. Let us therefore state three
convenient results which only bound this term, and then define it
precisely.
\begin{thm}
\label{thm:easy}Let $w$ be an arbitrary weight function. Then
\begin{equation}
\Delta_{w}\ge\frac{c\delta}{\lmix w}\frac{\min_{i}w_{i}^{2}}{w_{\mathrm{tot}}}\Delta_{K_{n}}\label{eq:thm easy}
\end{equation}
where the factor $\delta$ has the following estimates:
\begin{enumerate}
\item $\delta\ge c\left(\frac{\min^{*}w_{ij}}{\max w_{i}}\right)^{2}$ where
$\min^{*}w_{ij}=\min\{w_{ij}:w_{ij}>0\}.$ 
\item $\delta\ge c$ when $w$ comes from a regular graph.
\item It always holds that $\delta\ge\frac{1}{2\lmix w}$.
\end{enumerate}
\end{thm}

We note that in the case covered by clause 2, i.e.\ when $w$ comes
from a regular graph with the degree of all the vertices $d$, then
also $\min w_{i}^{2}/w_{\textrm{tot}}=d/n$ and we get $\Delta_{w}\ge cd\Delta_{K_{n}}/(n\lmix w)$.
Here and below $c$ and $C$ stand for universal positive constants.
Their value could change from line to line or even within the same
line.

Let us now define $\delta$ precisely. It is given by
\begin{equation}
\frac{1}{\delta}=\prod_{k=0}^{\lfloor\log_{2}\lmix w\rfloor}\max_{i}(1+p_{2^{k}}(i,i)).\label{eq:def delta}
\end{equation}
where $p_{t}$ are as in the definition of our lazy mixing time.
\begin{thm}
\label{thm:main}(\ref{eq:thm easy}) holds with $\delta$ given by
(\ref{eq:def delta}).
\end{thm}

Clause 3 of theorem \ref{thm:easy} is an obvious corollary of theorem
\ref{thm:main} (just bound all probabilities by 1). Clauses 1 and
2 are not so difficult either, but we will show them later.

The last thing we wish to do before going to the proof of theorem
\ref{thm:main} is to recall the octopus inequality. It claims that
for any numbers $w_{1i}$ we have
\[
\sum_{2\le i\le n}w_{1i}\nabla_{1i}\ge\sum_{2\le i<j\le n}\frac{w_{1i}w_{1j}}{w_{1}}\nabla_{ij}
\]
 where the inequality is as operators on $L^{2}(S_{n})$, and we still
to use (\ref{eq:def tot}) so $w_{1}=\sum_{i=2}^{n}w_{1i}$. See \cite[theorem 2.3]{CLR10},
and perhaps also \cite[theorem 4.2]{C16} which has a simplified proof,
and a statement identical to ours. 

\begin{proof}
[Proof of theorem \ref{thm:main}] For the purpose of the proof, it
will be convenient to extend weight functions also to the case $i=j$.
Let $u$ be such a weight function, i.e.\ $u_{ij}\in[0,\infty)$
for $i\le j$ in $\{1,\dotsc,n\}$ and $u_{ij}=u_{ji}$. We define
$\nabla_{ii}=0$ so the $u_{ii}$ do not affect $\Delta_{u}$, but
we do redefine $u_{i}=\sum_{j}u_{ij}$ so the numbers $u_{ij}/u_{i}$
still add up to $1$ and can be thought of as transition probabilities
of a Markov chain, this time one that might stay in the same place
for a turn.

For any weight function $u$ we define a new weight function $u^{(2)}$
by
\[
u_{ij}^{(2)}=\sum_{k}\frac{u_{ik}u_{kj}}{u_{k}}.
\]
It is straightforward to check that $u_{i}^{(2)}=u_{i}$ and that
the corresponding Markov process is the same as doing two steps in
the Markov process of $u$. Our goal is to compare $\Delta_{u}$ and
$\Delta_{u^{(2)}}$.

For this purpose write 
\[
\Delta_{u}=\frac{1}{2}\sum_{i}\Big(\sum_{j\ne i}u_{ij}\nabla_{ij}\Big)
\]
and apply the octopus inequality to each term. We get
\begin{align*}
\Delta_{u} & \ge\frac{1}{2}\sum_{i}\sum_{i\ne j<k\ne i}\frac{u_{ij}u_{ik}}{u_{i}-u_{ii}}\nabla_{jk}\ge\frac{1}{2}\sum_{i}\sum_{i\ne j<k\ne i}\frac{u_{ij}u_{ik}}{u_{i}}\nabla_{jk}\\
 & =\frac{1}{2}\sum_{j<k}\nabla_{jk}\sum_{i\not\in\{j,k\}}\frac{u_{ji}u_{ik}}{u_{i}}.
\end{align*}
The sum over all $i$, without the restriction $i\not\in\{j,k\}$,
is our ``target'' $\Delta_{u^{(2)}}$ so we get
\begin{equation}
\Delta_{u}\ge\frac{1}{2}\Delta_{u^{(2)}}-\frac{1}{2}\sum_{j<k}\nabla_{jk}\sum_{i\in\{j,k\}}\frac{u_{ji}u_{ik}}{u_{i}}.\label{eq:octopus2}
\end{equation}
To estimate the remainder, define
\[
\epsilon=\max_{i}\frac{u_{ii}}{u_{i}}
\]
and get
\begin{equation}
\sum_{j<k}\nabla_{jk}\sum_{i\in\{j,k\}}\frac{u_{ji}u_{ik}}{u_{i}}\le2\epsilon\sum_{j<k}\nabla_{jk}u_{jk}=2\epsilon\Delta_{u}\label{eq:eps}
\end{equation}
Summing (\ref{eq:octopus2}) and (\ref{eq:eps}) we get
\begin{equation}
\Delta_{u^{(2)}}\le(2+2\epsilon)\Delta_{u}.\label{eq:w2}
\end{equation}
We now apply (\ref{eq:w2}) inductively. Define 
\[
u^{(2^{k})}=\left(u^{(2^{k-1})}\right)^{(2)}
\]
(so $u^{(2^{k})}$ corresponds to the Markov process that does $2^{k}$
steps at a time). Define
\[
\epsilon_{k}=\max_{i}\frac{u_{ii}^{(2^{k})}}{u_{i}^{(2^{k})}}.
\]
Then applying (\ref{eq:w2}) $k$ times gives
\begin{equation}
\Delta_{u^{(2^{k})}}\le2^{k}\Delta_{u}\prod_{i=0}^{k-1}(1+\epsilon_{k}).\label{eq:2^k}
\end{equation}
Recall that in the statement of the theorem we are given a weight
function $w_{ij}$, $i\ne j$. To apply (\ref{eq:2^k}) to this $w$,
first define 
\[
u_{ij}=\begin{cases}
w_{ij} & i\ne j\\
w_{i} & i=j
\end{cases}
\]
so that the Markov chain corresponding to $u$ is just the Markov
chain corresponding to $w$, with an extra $\frac{1}{2}$ laziness,
as in our definition of lazy mixing time. In particular
\[
\epsilon_{k}=\max_{i}p_{2^{k}}(i,i).
\]
Let $k$ be the minimal such that $2^{k}\ge\lmix w$. By our definition
of the lazy mixing time, after $2^{k}$ steps the probability to pass
from every $i$ to every $j$ is at least $\frac{3}{4}$ the stationary
measure, i.e.\ $\tfrac{3}{4}w_{j}/w_{\text{tot}}$. Using the equality
$u_{i}^{(2^{k})}=\linebreak[0]u_{i}=2w_{i}$ we get
\[
u_{ij}^{(2^{k})}\ge\tfrac{3}{2}\frac{w_{i}w_{j}}{w_{\text{tot}}}\ge\frac{\min w_{i}^{2}}{w_{\textrm{tot}}}\quad\forall i,j.
\]
or
\[
\Delta_{u^{(2^{k})}}\ge\frac{\min w_{i}^{2}}{w_{\text{tot}}}\Delta_{K_{n}}.
\]
With (\ref{eq:2^k}) we get
\[
\Delta_{u}\ge2^{-k}\prod_{i=0}^{k-1}(1+\epsilon_{k})^{-1}\frac{\min w_{i}^{2}}{w_{\text{tot}}}\Delta_{K_{n}}.
\]
We have $2^{-k}>1/(2\lmix w)$ and $\prod(1+\epsilon_{k})^{-1}\ge\delta$
(in fact they are equal except when $\log_{2}\lmix w$ is integer,
in which case $\delta$ has one more term). The theorem is thus proved.
\end{proof}

The proof of clauses 1 and 2 of theorem \ref{thm:easy} uses the well-known
connection between isoperimetric inequalities and transition probabilities
of random walk. Let us formulate it as a lemma.
\begin{lem}
\label{lem:connectivity}For any connected weight function $w$, any
$t\le\lmix w$ and any vertices $i$ and $j$ we have
\[
p_{t}(i,j)\le\frac{30}{\sqrt{t}}\cdot\frac{w_{i}}{\min^{*}w_{ij}}.
\]
\end{lem}

(a weight function $w$ is called connected if for any $i$ and $j$
there exists some $t$ such that $p_{t}(i,j)>0$. As always, $p_{t}$
are the transition probabilities after adding to $w$ $\frac{1}{2}$-laziness)
\begin{proof}
As already mentioned, we use the connection between isoperimetric
inequalities and transition probabilities of random walk, \cite{MP05}
will serve as a convenient reference. Here are the necessary definitions:
Let $p_{t}(\cdot,\cdot)$ be the transition probabilities of a finite
Markov chain with laziness at least $\frac{1}{2}$, and let $\pi$
be its stationary measure. For a set $S$ of elements of the chain
define
\[
\pi(S)=\sum_{i\in S}\pi(i)\qquad|\partial S|=\sum_{i\in S,j\not\in S}p(i,j)\pi(i).
\]
Since $w$ may be scaled without changing the result, let us assume
that $w_{\textrm{tot}}=1$ which means that $w_{ij}=p(i,j)\pi(i)$,
and $w_{i}=\pi(i)$. Define the isoperimetric profile function $\phi:[0,\frac{1}{2}]\to[0,\infty)$
by
\[
\phi(r)=\inf\Big\{\frac{|\partial S|}{\pi(S)}:S\text{ such that }\pi(S)\le r\Big\}
\]
and extend $\phi$ beyond $\frac{1}{2}$ by defining $\phi(r)=\phi(\frac{1}{2})$
for all $r>\frac{1}{2}$. Then theorem 1 of \cite{MP05} states that
for every $\lambda>0$, if 
\begin{equation}
t\ge1+\int_{4\min(\pi(i),\pi(j))}^{4/\lambda}\frac{4dr}{r\phi^{2}(r)}\label{eq:felini-1}
\end{equation}
then 
\begin{equation}
\Big|\frac{p_{t}(i,j)-\pi(j)}{\pi(j)}\Big|\le\lambda.\label{eq:mixing-1}
\end{equation}
(if you try to compare to \cite{MP05}, you will notice that they
denote by $\epsilon$ what we denote by $\lambda$ \textemdash{} but
have no fear, their results are not restricted to small $\lambda$
in any way).

Since our only assumption is connectivity, the only lower bound we
can give on $\phi$ is
\[
\phi(r)\ge\frac{\min^{*}w_{ij}}{r}
\]
which allows to bound the integral by
\[
\int_{4\min(\pi(x),\pi(y))}^{4/\lambda}\frac{4dr}{r\phi^{2}(r)}\le\frac{4}{(\min^{*}w_{ij})^{2}}\int_{0}^{4/\lambda}r\,dr=\frac{32}{(\lambda\min^{*}w_{ij})^{2}}.
\]
Set therefore 
\[
\lambda=\sqrt{\frac{32}{(t-1)\min^{*}w_{ij}^{2}}}\stackrel{(*)}{\le}\frac{6}{\sqrt{t+1}\min^{*}w_{ij}}
\]
where in $(*)$ we assume $t\ge17$, which we can, as for $t<17$
the lemma holds trivially (the stated bound for the probability is
bigger than one). This choice of $\lambda$ makes the integral $\le t-1$
and hence $t$ satisfies the requirement (\ref{eq:felini-1}) and
we may conclude (\ref{eq:mixing-1}). Writing (\ref{eq:mixing-1})
with the value of $\lambda$ above gives
\[
p_{t}(i,j)\le\left(\frac{6}{\sqrt{t+1}\min^{*}w_{ij}}+1\right)w_{i}.
\]
To get rid of the annoying $+1$ factor, argue as follows: if $6/(\sqrt{t+1}\min^{*}w_{ij})\linebreak[0]<1/4$
then by (\ref{eq:mixing-1}) we must have $t+1\ge\lmix w$ and we
assumed this is not the case. Otherwise we have $1\le24/(\sqrt{t+1}\min^{*}w_{ij})$
and can write
\[
p_{t}(i,j)\le\frac{30}{\sqrt{t+1}}\cdot\frac{w_{i}}{\min^{*}w_{ij}}<\frac{30}{\sqrt{t}}\cdot\frac{w_{i}}{\min^{*}w_{ij}}
\]
as promised.
\end{proof}

\begin{proof}
[Proof of clause 1 of theorem \ref{thm:easy}]Let us first get an
uninteresting case off our hands: the case where $w$ is disconnected,
i.e.\ when there are $i$ and $j$ such that there is no paths of
positive $w$ between them. In this case $\lmix w=\infty$ and theorem
\ref{thm:easy} is trivially true. (Do note that without the assumption
of connectedness the $\delta$ defined by (\ref{eq:def delta}) is
$\infty$, so this assumption is indeed necessary) 

Since $w$ is connected we may apply lemma \ref{lem:connectivity}.
Examine the product 
\[
\frac{1}{\delta}=\prod_{k=0}^{\lfloor\log_{2}\lmix w\rfloor}\max_{i}(1+p_{2^{k}}(i,i)).
\]
For the first $10+2\log_{2}\frac{\max w_{i}}{\min^{*}w_{ij}}$ terms
the bound of lemma \ref{lem:connectivity} is not better than the
trivial bound $p\le1$, so we use it and get a factor of $1024\left(\frac{\max w_{i}}{\min^{*}w_{ij}}\right)^{2}$.
In the rest of the terms the $p_{2^{k}}$ decay exponentially, so
the product is bounded by a constant. This proves the claim.
\end{proof}
\begin{lem}
\label{lem:jonathan}Let $G$ be a regular, connected graph, let $i$
and $j$ be vertices and let $t\le\lmix G$. Then
\[
p_{t}(i,j)\le\frac{C}{t^{1/4}}.
\]
\end{lem}

\begin{proof}
Denote the common degree of the vertices of $G$ by $d$. We divide
into two cases, according to whether $d\ge t^{1/4}$ or not. If $d\ge t^{1/4}$
then we note that the function $\max_{i,j}p_{t}(i,j)$ is decreasing
in $t$ (we use here that the graph is regular), and that $p_{1}(i,j)\le d^{-1}\le t^{-1/4}$.
If $d<t^{1/4}$ we use lemma \ref{lem:connectivity} and get that
\[
|p_{t}(i,j)|\le\frac{30}{\sqrt{t}}d<\frac{30}{t^{1/4}}.\qedhere
\]
\end{proof}

\begin{proof}
[Proof of clause 2 of theorem \ref{thm:easy}]As in the proof of clause
1, we may assume $G$ is connected. Write, using lemma \ref{lem:jonathan},
\[
\frac{1}{\delta}=\prod_{k=0}^{\lfloor\log_{2}\lmix w\rfloor}\max_{i}(1+p_{2^{k}}(i,i))\le\prod_{k=0}^{\infty}\Big(1+\frac{C}{2^{k/4}}\Big)\le C.\qedhere
\]
\end{proof}
We finish this section with a comparison of $\lmix$ with the more
common notion of mixing time, the total variation mixing time. Denote
\begin{equation}
\mix w=\min\{t:\forall i\;||p_{t}(i,\cdot)-\pi||_{\text{TV}}<\tfrac{1}{4}\}\label{eq:def tvmix}
\end{equation}
where $||\cdot||_{\textrm{TV}}$ is the total variation norm (or $\frac{1}{2}$
of the $l^{1}$ norm).
\begin{lem}
\label{lem:hashevakh lealdous}For reversible Markov chains, $\frac{1}{8}\lmix w\le\mix w\le\lmix w$.
\end{lem}

\begin{proof}
To show $\mix\le\lmix$, choose some $t$ such that $p_{t}(i,j)>\frac{3}{4}\pi(j)$
for all $i$ and $j$ and get 
\[
||p_{t}(i,\cdot)-\pi||_{\textrm{TV}}=\sum_{j}(\pi(j)-p_{t}(i,j))^{+}<\sum_{j}\tfrac{1}{4}\pi(j)=\tfrac{1}{4}.
\]
(this direction does not require reversibility).

In the other direction, assume that at some $t_{0}$ we have $\max_{i}||p_{t_{0}}(i,\cdot)-\pi||_{\text{TV}}=:d_{t}<\frac{1}{4}$.
It is well known that $d_{t}$ is monotonically decreasing, and further
that $d_{t+s}\le2d_{t}d_{s}$ \cite[lemmas 4.11 and 4.12]{LPW09}.
Applying this twice gives for all $t\ge4t_{0}$ that $d_{t}<\frac{1}{32}$.
We write
\[
p_{2t}(i,j)=\sum_{k}p_{t}(i,k)p_{t}(k,j).
\]
Reversibility means that $\pi(j)p_{t}(j,k)=\pi(k)p_{t}(k,j)$ so we
may continue the calculation to get
\begin{equation}
p_{2t}(i,j)=\pi(j)\sum_{k}p_{t}(i,k)p_{t}(j,k)/\pi(k).\label{eq:2t t}
\end{equation}
Using Cauchy-Schwarz gives 
\begin{multline}
\sum_{k}\sqrt{p_{t}(i,k)p_{t}(j,k)}=\sum_{k}\sqrt{\frac{p_{t}(i,k)p_{t}(j,k)}{\pi(k)}}\cdot\sqrt{\pi(k)}\\
\le\Big(\sum_{k}\frac{p_{t}(i,k)p_{t}(j,k)}{\pi(k)}\Big)^{1/2}\Big(\sum_{k}\pi(k)\Big)^{1/2}\label{eq:CS}
\end{multline}
and the term $\sum\pi(k)$ is of course 1 so we get
\begin{align*}
\sqrt{\frac{p_{2t}(i,j)}{\pi(j)}}\, & \stackrel{\textrm{\clap{(\ref{eq:2t t})}}}{=}\Big(\sum_{k}p_{t}(i,k)p_{t}(j,k)/\pi(k)\Big)^{1/2}\\
 & \stackrel{\textrm{\ensuremath{\smash{{\clap{(\ref{eq:CS})}}}}}}{\ge}\sum p_{t}(i,k)^{1/2}p_{t}(j,k)^{1/2}\\
 & \ge\sum\min(p_{t}(i,k),p_{t}(j,k))\\
 & =\sum\frac{p_{t}(i,k)+p_{t}(j,k)}{2}-\frac{|p_{t}(i,k)-p_{t}(j,k)|}{2}\\
 & =1-||p_{t}(i,\cdot)-p_{t}(j,\cdot)||_{\textrm{TV}}\\
 & \ge1-||p_{t}(i,\cdot)-\pi||_{\textrm{TV}}-||\pi-p_{t}(j,\cdot)||_{\textrm{TV}}>\frac{15}{16}.
\end{align*}
All in all we get $p_{t}(i,j)>\frac{225}{256}\pi(j)$ for all $t>8t_{0}$,
proving the lemma.
\end{proof}
For every possible equivalent formulation of the mixing time, including
a few quite similar (but not identical) to ours, see \cite{LW}.

\section{Corollaries}

\subsection{The mixing time}

We start with a corollary for the mixing time. Our operator $\Delta_{w}$
is most suitable for studying random walk in \emph{continuous time}
because it is the generator of this walk (do not be fooled by the
intensive use of discrete time random walk during the proof of theorem
\ref{thm:main}, and from the discrete time formulation of $\delta$
\textemdash{} this is an artifact of the proof method). So let us
recall what is the continuous time interchange process. For a given
weight function $w$ we define a stochastic process on $S_{n}$ in
continuous time by 
\[
p_{t}(i,j)=e^{-t\Delta_{w}}(i,j)
\]
where $e^{-t\Delta}$ is matrix exponentiation (and the matrix is
of course $n!\times n!$). It is well-known that this is a Markov
chain in continuous time, and further, it is equivalent to the following
process: Put marbles on $\{1,\dotsc,n\}$, all different, and Poisson
clocks ``on'' each couple $ij$, with the clock on the edge $ij$
having rate $w_{ij}$, and when the clock rings exchange the two marbles.
This is known as the interchange process corresponding to the weight
function $w$. Mixing time for continuous time random walks is traditionally
defined using the total variation distance, i.e.\ by (\ref{eq:def tvmix}).
(The assumption of laziness we had in discrete time is no longer necessary).
Denote by $\imix w$ the mixing time of the interchange process corresponding
to $w$.
\begin{lem}
If $\Delta_{w}\ge a\Delta_{K_{n}}$ then 
\[
\imix w\le\frac{C}{a}\frac{\log n}{n}.
\]
\end{lem}

\begin{proof}
This result is an easy corollary of a comparison argument which can
be found in equation (2.10) in \cite{DSC93}, and the determination
of $\imix K_{n}$ by Diaconis and Shahshahani \cite{DS81}. Let us
give some details. First we need to translate the inequality $\Delta_{w}\ge a\Delta_{K_{n}}$
to exponentials. Now, in general one cannot conclude from an operator
inequality $A\ge B$ that $e^{-A}\le e^{-B}$ (see \cite{K56} for
more on this), but in our case the operators $\Delta_{w}$ and $\Delta_{K_{n}}$
commute. Indeed, $\Delta_{K_{n}}$ is in the center of $S_{n}[\mathbb{C}]$
and commutes with any convolution operator on $S_{n}$. For commuting
operators, number inequalities translate directly to operator inequalities
(simply by taking the joint diagonalisation), so $\Delta_{w}\ge a\Delta_{K_{n}}$
implies $e^{-s\Delta_{\omega}}\le e^{-sa\Delta_{K_{n}}}$ for any
$s>0$.

Let us now explain the notations of \cite{DSC93}. Their setting is
that of discrete time random walk on $S_{n}$, so we will apply their
results for the walk that does steps of $e^{s\Delta_{w}}$ for some
small $s>0$. Quoting \cite[(2.10)]{DSC93} literally, it states that
``if $\mathcal{\tilde{E}}\le2\mathcal{E}$ and $\tilde{\mathcal{F}}\le2\mathcal{F}$
then for all $k$, $||p_{k}-\pi||^{2}\le e^{-k/2}+||\tilde{p}_{k/4}-\pi||^{2}$''.
The notations $\mathcal{E}$ and $\tilde{\mathcal{E}}$ stand for
our operators $1-e^{-s\Delta_{w}}$ and $1-e^{-sa\Delta_{K_{n}}}$
so $\tilde{\mathcal{E}}\le\mathcal{E}$ without the need for the $2$.
The notations $\mathcal{F}$ and $\tilde{\mathcal{F}}$ stand for
the operators $1+e^{-s\Delta_{w}}$ and $1+e^{-sa\Delta_{K_{n}}}$.
In the setting of \cite{DSC93} one has to care also about periodicity,
but in our setting this is automatic: $\mathcal{F}$ and $\tilde{\mathcal{F}}$
are both bounded between 1 and 2, so $\tilde{\mathcal{F}}\le2\le2\mathcal{F}$.
So the conditions are satisfied. $p_{k}$ and $\tilde{p}_{k}$ are
the distributions after $k$ steps of the discrete walk, which is
just the distributions of our continuous time walk after $sk$ steps.
So we get
\[
||p_{sk}^{w}-\pi||^{2}\le e^{-k/2}+||p_{ask/4}^{K_{n}}-\pi||^{2}.
\]
Taking $s\to0$ and $k\to\infty$ while preserving $t=sk$ gives
\[
||p_{t}^{w}-\pi||^{2}\le||p_{at/4}^{K_{n}}-\pi||^{2}
\]
which of course implies that 
\[
\imix w\le\frac{4}{a}\imix K_{n}.
\]
Finally \cite{DS81} shows that $\imix K_{n}\approx\frac{\log n}{n}$,
proving the lemma (the number stated in \cite{DS81}, $\frac{1}{2}n\log n$,
has to be divided by $\binom{n}{2}$, the number of edges of $K_{n}$,
because of, yet again, the difference between discrete and continuous
time).
\end{proof}
\begin{cor}
\label{cor:mixing}For every weight function $w$,
\[
\imix w\le\frac{\lmix w}{c\delta}\frac{w_{\mathrm{tot}}}{n\min w_{i}^{2}}\log n.
\]
with $\delta$ given by either (\ref{eq:def delta}) or its estimates
in theorem \ref{thm:easy}.
\end{cor}

\subsection*{The chameleon process}

Invented by Ben Morris \cite{M06}, the chameleon process has been
used with great success to bound mixing times of interacting particle
systems \cite{O13,HP}. Let us therefore compare the two approaches.
This section is somewhat geared towards mixing time aficionados, so
some key terms will not be fully explained.

The approach of this paper has two advantages when compared to the
chameleon process. The first is that the chameleon process only bounds
the mixing time of the exclusion process. The exclusion process is
similar to the interchange process, but the marbles have only two
colours, so it is a process on a state space of size $\binom{n}{k}$
for some $k\in\{1,\dotsc,n\}$. The approach works even with multiple
colours, but always some proportion of the marbles (i.e.\ at least
$cn$ for some $c>0$) must be indistinguishable. The second advantage
is that, in fact, we have bounded the $l^{2}$ mixing time and not
just the total variation mixing time (see e.g.\ \cite[definition 2.2]{B16}
for the $l^{2}$ mixing time). We formulated corollary \ref{cor:mixing}
for the total variation time because this is how Diaconis and Shahshahani
formulate their result for the interchange process on $K_{n}$ (and,
of course, the total variation time is the most common notion of mixing
time in the literature, often called \emph{the} mixing time). But
in fact they bound the $l^{2}$ mixing time (see their \cite[formula (3.1)]{DS81}),
and this was already noted in the literature \cite[theorem 6.9]{B16}.
Further, an operator inequality implies a comparison of $l^{2}$ mixing
time \cite[theorem 4.3]{B16}. The chameleon process bounds the total
variation time, and it does not seem easy to change this fact.

Contrariwise, some of the results using the chameleon process give
better bounds than ours. Examine, for example, the results of Oliveira
\cite{O13}. Denote by $\ctvmix w$ the mixing time of continuous
time random walk with respect to the weight function $w$, and by
$\ectvmix w$ the same thing but for exclusion process with weight
function $w$ (i.e.\,again in continuous time and with respect to
the total variation distance). Then Oliveira showed that
\[
\ectvmix w\le C\ctvmix w\log n
\]
without all our extra factors. Further, Hermon and Pymar show \cite{HP}
that this bound is not always tight. For example, for an expander
they show $\ectvmix w\le\log n\log\log n$, again using the chameleon
process.

We finish this section with a few strengthenings of results of \cite{J12},
which studied the interchange process using operator comparisons (and
not via the chameleon process). Since the results of \cite{J12} are
formulated in discrete time, one has to scale by the total number
edges to compare: in corollaries \ref{cor:tree} and \ref{cor:expander}
multiply by $n$ and in corollary \ref{cor:hypercube} by $n\log n$
to get discrete time results. In corollary \ref{cor:tree} our bound
matches the lower bound given in \cite{J12}.
\begin{cor}
\label{cor:tree}For $G$ being a finite regular tree of degree $d$,
\[
\imix G\le C(d)n\log n.
\]
\end{cor}

\begin{proof}
The mixing time of a regular tree is of order $n$, see \cite[example 5.14]{AF}.
\end{proof}
\begin{cor}
\label{cor:expander}For $G$ being a bounded degree expander, $\imix G\le C\log^{2}n$.
\end{cor}

\begin{cor}
\label{cor:hypercube}For $G=(\mathbb{Z}/2)^{d}$ (a.k.a.\ the hypercube),
\[
\imix G\le C\log n\log\log n.
\]
\end{cor}

\begin{proof}
The discrete-time mixing time of $(\mathbb{Z}/2)^{d}$ is $d\log d$,
see \cite[example 5.15, equation (5.71)]{AF}. We bound $\delta$
using clause 2 of theorem \ref{thm:easy}.
\end{proof}
For the hypercube, Wilson proved a lower bound of $\log n$ and conjectured
that it is the correct value \cite[\S 9.1]{W04}. For the exclusion
process Wilson's conjecture was proved recently in \cite{HP}.

\subsection{\label{subsec:cycles}Appearance of large cycles}

We now leave the question of when the interchange process mixes completely
and study a different question: when do large cycles first appear?
Here is the corresponding comparison result:
\begin{thm}
\label{thm:Bati}If $\Delta_{w}\ge a\Delta_{K_{n}}$ then the interchange
process at time $\ge C/an$ satisfies 
\[
\mathbb{P}(\exists\text{cycle of length}>n/2)>c.
\]
\end{thm}

\begin{proof}
The proof uses the representation theory of $S_{n}$. As this note
is quite short it would not be possible to give all the necessary
background, consult \cite{AK13} for connections specific to the interchange
process, or any textbook on the subject, e.g.\ \cite{JK81}. In short,
an irreducible representation $\rho$ is a subspace of $L^{2}(S_{n})$
which is preserved by any convolution operator, in particular by $\Delta_{w}$
and $\Delta_{K_{n}}$ (and is minimal in that regard). The restriction
of $\Delta_{w}$ to the subspace $\rho$ is diagonalisable (recall
that $\Delta_{w}$ is self-adjoint). Let us denote its eigenvalues
by 
\[
\lambda_{1}(w,\rho)\le\dotsb\le\lambda_{\dim\rho}(w,\rho).
\]
In the case of the complete graph, $\Delta_{K_{n}}$ restricted to
$\rho$ is a scalar matrix (this is an easy corollary of Schur's lemma),
and its value was calculated by Diaconis and Shahshahani \cite{DS81}.
In other words $\lambda_{i}(K_{n},\rho)$ are all equal \textemdash{}
let us denote the common value by $\lambda(K_{n},\rho)$. The inequality
$\Delta_{w}\ge a\Delta_{K_{n}}$ restricts, of course, to any irreducible
representation, and from it and the scalarity of $\Delta_{K_{n}}$
we get
\begin{equation}
\lambda_{i}(w,\rho)\ge a\lambda(K_{n},\rho).\label{eq:in one rep}
\end{equation}
The final element in the proof is a formula connecting $\lambda_{i}(w,\rho)$
with the probability of large cycles \cite{AK13}. To state it, let
us recall that irreducible representations of $S_{n}$ are indexed
by partitions of $n$, i.e.\ numbers $\sigma_{1}\ge\dotsb\ge\sigma_{r}$
such that $\sum\sigma_{i}=n$. We denote by $[\sigma_{1},\dotsc,\sigma_{r}]$
the irreducible representation corresponding to the partition $\sigma$,
and if a certain number repeats more than once we denote it by a superscript,
namely $[3,1^{3}]$ is the partition $6=3+1+1+1$. We may now state
the results of \cite{AK13}. Fix $k\in\{1,\dotsc,n\}$ and denote
by $s_{k}(t)$ the number of cycles of length $k$ in the interchange
process corresponding to $w$ at time $t$. Then lemma 5 of \cite{AK13}
states that
\begin{equation}
\mathbb{E}(s_{k}(t))=\frac{1}{k}\sum_{\rho}a_{\rho}\sum_{j=1}^{\dim\rho}e^{-t\lambda_{j}(w,\rho)}\label{eq:sk decomposition}
\end{equation}
where
\[
a_{\rho}=\begin{cases}
1 & \rho=[n]\\
(-1)^{i+1} & \rho=[k-i-1,n-k+1,1^{i}]\text{ for }i\in\{0,\dotsc,2k-n-2\}\\
(-1)^{i} & \rho=[n-k,k-i,1^{i}]\text{ for }i\in\{\max\{2k-n,0\},\dotsc,k-1\}\\
0 & \text{otherwise.}
\end{cases}
\]
Now, $[n]$ is the one dimensional space of constant functions, so
$\lambda_{1}(w,[n])=0$. Our proof strategy will be to show that the
corresponding term in the sum, $\frac{1}{k}$, is the main term and
all the others are negligible. Hence we need estimates for $\dim\rho$
and for $\lambda(K_{n},\rho)$. Both go back to the early 20th century,
but a convenient reference would be \cite{BK15}. From \cite[\S 2.3]{BK15}
we get:
\begin{align*}
\lambda(K_{n},[k-i-1,n-k+1,1^{i}]) & =\binom{n}{2}+ik+k-\frac{1}{2}((n-k)^{2}+k^{2}-n)\\
\dim([k-i-1,n-k+1,1^{i}]) & =\frac{n!(2k-n-i-1)}{i!k(n-k)!(k-i-1)!(n-k+i+1)}
\end{align*}
while from \cite[equations 8 \& 9]{BK15} we get essentially the same
formulas for the second family of representations
\begin{align*}
\lambda(K_{n},[n-k,k-i,1^{i}]) & =\binom{n}{2}+ik+k-\frac{1}{2}((n-k)^{2}+k^{2}-n)\\
\dim([n-k,k-i,1^{i}]) & =\frac{n!(n-2k+i+1)}{i!k(n-k)!(k-i-1)!(n-k+i+1)}
\end{align*}
(a forthcoming paper of ours \cite{AK} will partially explain this
similarity). We see that in both cases we have that the dimension
is $\binom{n}{k}\binom{k}{i}$ times a rational factor which can be
bounded roughly by $n^{2}$. So $\dim\rho\le4^{n}n^{2}$ for all relevant
representations. As for $\lambda$, we restrict our attention to $k\in[\frac{1}{2}n,\frac{3}{4}n]$
and get that in all cases $\lambda\ge cn^{2}$. This means that for
any $\rho\ne[n]$ for which $a_{\rho}\ne0$ we have 
\[
\sum_{j=1}^{\dim\rho}\exp(-t\lambda_{j}(w,\rho))\stackrel{\textrm{(\ref{eq:in one rep})}}{\le}\sum_{j=1}^{\dim\rho}\exp(-at\lambda(K_{n},\rho))\le4^{n}n^{2}\exp(-actn^{2})
\]
so
\[
\Big|\mathbb{E}(s_{k}(t))-\frac{1}{k}\Big|=\Big|\sum_{\rho\ne[n]}a_{\rho}\sum_{j=1}^{\dim\rho}e^{-t\lambda_{j}(w,\rho)}\Big|\le4^{n}n^{3}\exp(-actn^{2}).
\]
Taking $t>C/an$ for some $C$ sufficiently large makes the right
hand side negligible, and we get that $\mathbb{E}(s_{k}(t))\ge\frac{1}{2k}$,
for all $k\in[\frac{1}{2}n,\frac{3}{4}n]$. Since $k>\frac{1}{2}n$,
$s_{k}$ may only take the values 0 and $1$ and $\mathbb{E}(s_{k})=\mathbb{P}(s_{k}=1)$.
Summing over $k\in[\frac{1}{2}n,\frac{3}{4}n]$ gives the theorem. 
\end{proof}
\begin{cor}
For the 2 dimensional Hamming graph, at time $t>C/\sqrt{n}$ we have
cycles larger than $n/2$ with positive probability.
\end{cor}

The interchange process on the Hamming graph was studied in \cite{MS16,AKM}
using different methods.
\begin{proof}
The 2 dimensional Hamming graph is the graph given by, for $n=m^{2}$,
\[
w_{((i_{1},i_{2}),(j_{1},j_{2}))}=\begin{cases}
1 & i_{1}=j_{1}\\
1 & i_{2}=j_{2}\\
0 & \text{otherwise}
\end{cases}
\]
(recall that $w$ is not defined for $(i_{1},i_{2})=(j_{1},j_{2})$).
This graph has finite mixing time, $w_{i}=2(m-1)$ and $w_{\textrm{tot}}=2(m^{3}-m^{2})$
so by theorem \ref{eq:thm easy} $\Delta_{w}\ge\frac{c}{m}\Delta_{K_{n}}$.
By theorem \ref{thm:Bati}, we get large cycles at time $C/\sqrt{n}$,
as claimed.
\end{proof}
Theorem \ref{thm:Bati} has an analog for the quantum Heisenberg ferromagnet
(QHF for short). Rather than defining the QHF, we apply T\'oth's
representation \cite{Toth}, and get that the magnetisation of the
QHF can be calculated by considering the cycles of the interchange
process with a certain weighting. Precisely, let $\alpha_{k}(t)$
be the number of cycles of length $k$ at time $t$ and let $\alpha(t)=\sum_{k}\alpha_{k}(t)$.
Then the magnetisation, $m$, is defined by
\[
Z(t)=\mathbb{E}(2^{\alpha(t)})\qquad m^{2}(t)=\frac{1}{Z(t)}\mathbb{E}\Big(\Big(\sum_{k}k^{2}\alpha_{k}(t)\Big)2^{\alpha(t)}\Big).
\]
 (For the relation to physics we refer to \cite{Toth}, but basically
it is a quantum model with spin interactions. The graph over which
we perform the interchange process describes the inter-particle interactions,
and the time $t$ translates to the inverse of the temperature).

Theorem \ref{thm:QHF} below is the analog of theorem \ref{thm:Bati}
for the quantum Heisenberg ferromagnet. It is proved identically,
replacing the results of \cite{BK15} with the results of \cite{AK},
at this time still in preparations. 
\begin{thm}
\label{thm:QHF}If $\Delta_{w}\ge a\Delta_{K_{n}}$ then the quantum
Heisenberg ferromagnet at $t\ge C/an$ satisfies $m\ge cn$.
\end{thm}

\begin{proof}
[Proof sketch] We keep the notations $\lambda_{i}(w,\rho)$ and $\lambda(K_{n},\rho)$
from the proof of theorem \ref{thm:Bati}, and recall that $\lambda_{i}(w,\rho)\ge a\lambda(K_{n},\rho)$.
We start with the partition function $Z$, for which we write a simple
analog of (\ref{eq:sk decomposition}):
\[
Z=\sum_{\rho=[a,b]}(a-b+1)\sum_{j=1}^{\dim\rho}e^{-t\lambda_{j}(w,\rho)}
\]
(the sum includes also $\rho=[n]$ for which we consider $a=n$ and
$b=0$. Ditto below, when we write $[a,b]$ we always entertain the
possibility that $b=0$). For $m$ we have a more complicated formula,
each $\alpha_{k}$ has 
\[
\mathbb{E}(\alpha_{k}2^{\alpha})=\sum_{\rho}d_{\rho,k}\sum_{j=1}^{\dim\rho}e^{-t\lambda_{j}(w,\rho)}
\]
with some coefficients $d_{\rho,k}$, which are described by \cite[theorem 3]{AK}.
Here we will only note a few properties of the $d_{\rho,k}$:
\begin{enumerate}
\item $d_{\rho,k}=0$ unless $\rho=[a,b,c,1^{d}]$ i.e.\ has (at most)
3 rows and one column;
\item $d_{[a,b],k}=\frac{2(a-b+1)}{k}$ for all $a$ and $b$ with $a+b=n$,
$0\le b\le\lfloor(n-k)/2\rfloor$;
\item $|d_{\rho,k}|\le2n+2$ for all $\rho$; and
\item If $k\in[\frac{1}{2}n,\frac{3}{4}n]$ then for all $\rho\ne[a,b]$
for which $d_{\rho,k}\ne0$ we have $\lambda(K_{n},\rho)>cn^{2}$.
\end{enumerate}
In particular, the first property allows to bound the dimensions of
the relevant representations. It is not difficult to conclude from
the hook formula (see e.g.\ \cite{JK81}) that $\dim\rho\le6^{n}$
for all $\rho$ of the form $[a,b,c,1^{d}]$ (the precise value 6
will play no role). 

Assume now $k\in[\frac{1}{2}n,\frac{3}{4}n]$. We make two estimates,
the first for $Z$. We note that for $[a,b]$ with $b>\lfloor(n-k)/2\rfloor$
we have $\lambda(K_{n},[a,b])>cn^{2}$ and hence the contribution
of these representations to $Z$ is negligible, namely denote
\[
Z'=\sum_{\substack{\rho=[a,b]\\
b\le\lfloor(n-k)/2\rfloor
}
}(a-b+1)\sum_{j=1}^{\dim\rho}e^{-t\lambda_{j}(w,\rho)}
\]
and get
\[
|Z-Z'|=\sum_{\substack{\rho=[a,b]\\
b>\lfloor(n-k)/2\rfloor
}
}(a-b+1)\sum_{j=1}^{\dim\rho}e^{-t\lambda_{j}(w,\rho)}\le n^{2}6^{n}\exp(-actn^{2})
\]
(the $n^{2}$ term has one $n$ bounding the number of values of $b$,
and one $n$ as a bound for $(a-b+1)$). Hence for $t>C/an$ we get
that $Z-Z'$ is negligible (recall that $Z\ge n+1$, the contribution
of the representation $[n]$).

As for $m$, we write 
\[
E'(k)=\sum_{\substack{\rho=[a,b]\\
b\le\lfloor(n-k)/2\rfloor
}
}d_{\rho,k}\sum_{j=1}^{\dim\rho}e^{-t\lambda_{j}(w,\rho)}
\]
and similarly get 
\[
|\mathbb{E}(\alpha_{k}2^{\alpha})-E'(k)\big|\le\sum_{\rho\ne[a,b]}|d_{\rho,k}|\sum_{j=1}^{\dim\rho}e^{-t\lambda_{j}(w,\rho)}\le C\cdot6^{n}n^{4}\exp(-actn^{2}).
\]
(we used here the third property of the $d_{\rho,k}$ to bound them
by $Cn$, and the fourth property to estimate $\lambda_{j}(w,\rho)\ge a\lambda(K_{n},\rho)\ge acn^{2}$).
Hence for $t>Ca/n$ this is also negligible. But
\begin{equation}
\frac{E'}{Z'}=\frac{2}{k}.\label{eq:didnt expect that, did you}
\end{equation}
Hence
\begin{align*}
m^{2} & =\frac{1}{Z}\sum_{k=1}^{n}k^{2}\mathbb{E}(\alpha_{k}2^{\alpha})\ge\frac{1}{Z}\sum_{k=\frac{1}{2}n}^{\frac{3}{4}n}k^{2}\mathbb{E}(\alpha_{k}2^{\alpha})\\
 & \ge\frac{1}{Z'}\sum_{k=\frac{1}{2}n}^{\frac{3}{4}n}k^{2}E'(k)\cdot(1+O(e^{-cn}))\\
 & \stackrel{\textrm{{\clap{(\ref{eq:didnt expect that, did you})}}}}{=}\sum_{k=\frac{1}{2}n}^{\frac{3}{4}n}2k(1+O(e^{-cn}))\ge cn^{2},
\end{align*}
as promised.
\end{proof}
\begin{cor}
The quantum Heisenberg ferromagnet on the 2 dimensional Hamming graph,
at $t>C/\sqrt{n}$ has $m\ge cn$.
\end{cor}

\subsection{A relation with the Caputo-Liggett-Richthammer theorem}

Finally, let us close cycles by returning to Aldous' conjecture. We
last mentioned Aldous' conjecture in the introduction, as its proof
was the first application of the octopus inequality. Equipped with
the notation $\lambda_{i}(w,\rho)$ of \S \ref{subsec:cycles} we
may now state it (which we did not do in the introduction) as follows:
\begin{equation}
\lambda_{1}(w,\rho)\ge\lambda_{1}(w,[n-1,1])\qquad\forall w,\forall\rho\ne[n].\label{eq:Aldous}
\end{equation}
The notations $[n]$ and $[n-1,1]$ are the standard notations for
presentations we also used in \S \ref{subsec:cycles}. We remark,
though, that these specific two are very simple: $[n]$ is the trivial,
one-dimensional representation while $[n-1,1]$ is the $(n-1)$-dimensional
representation one gets by removing the constants from the standard
representation, so $\lambda_{1}(w,[n-1,1])$ is simply the second
eigenvalue of the standard representation, i.e.\ of the continuous
time Markov chain corresponding to $w$. Operator comparison arguments
imply eigenvalue inequalities. Indeed, $\Delta_{w}\ge a\Delta_{K_{n}}$
is equivalent to $\lambda_{1}(w,\rho)\ge a\lambda(K_{n},\rho)$ for
all $\rho$. Thus theorem \ref{thm:easy} may be reformulated as
\begin{equation}
\lambda_{1}(w,\rho)\ge\frac{c\delta}{\lmix w}\frac{\min_{i}w_{i}^{2}}{w_{\textrm{tot}}}\lambda(K_{n},\rho)\qquad\forall w,\rho.\label{eq:lambda1 main}
\end{equation}
This begs the question: for which $\rho$ is this estimate better
than (\ref{eq:Aldous})? Let us restrict our attention to the case
that $w$ comes from a $d$-regular graph. In this case (\ref{eq:lambda1 main})
simplifies to
\[
\lambda_{1}(w,\rho)\ge\frac{cd}{n\lmix w}\lambda(K_{n},\rho).
\]
Further, it is generally true that $\lmix w\le(Cd\log n)/\lambda_{1}(w,[n-1,1])$,
see e.g.\ \cite[lemma 4.23]{AF} (let us explain the notations of
Aldous and Fill, for easier comparison: the relevant clause is the
continuous time inequality. Their $\tau_{2}$ is simply $1/\lambda_{1}$
and their $\pi^{*}$ is the minimum of the stationary measure, which
is in this case simply $1/n$. Finally the $d$ comes because our
$\lmix$ is the discrete mixing time, while their $\tau_{1}$ is the
continuous mixing time). We get
\[
\lambda_{1}(w,\rho)\ge\frac{c\lambda(K_{n},\rho)}{n\log n}\lambda_{1}(w,[n-1,1]).
\]
Thus we get a better estimate than (\ref{eq:Aldous}) for any $\rho$
for which $\lambda(K_{n},\rho)\ge Cn\log n$. This condition holds
for the vast majority of representations, for example for all those
with more than $C\log n$ boxes below the first row, see \cite[lemma 7]{DS81}
(again, for easier comparison, our $\lambda(K_{n},\rho)$ is $\binom{n}{2}(1-r(\rho))$
where $r$ is the character ratio calculated ibid.).

\subsection*{Acknowledgements}

We wish to thank Nick Crawford for noting the application to the Hamming
graph. We thank Jonathan Hermon for many interesting discussions,
and for lemma \ref{lem:jonathan}. GK is supported by the Israel Science
Foundation, by the Jesselson Foundation and by Paul and Tina Gardner.

\end{document}